\documentclass[a4paper, 11pt]{amsart}

\usepackage{latexsym}
\usepackage{amssymb,amsmath,mathtools}
\usepackage{graphics}
\usepackage{url}
\usepackage[vcentermath, enableskew]{youngtab}
  
\usepackage{tikz, tikz-cd}
\usepackage{booktabs}  

\newtheorem{theorem}{Theorem}[section]
\newtheorem*{theorem*}{Theorem}

\newtheorem{proposition}[theorem]{Proposition}

\newtheorem{lemma}[theorem]{Lemma}
\newtheorem*{lemma*}{Lemma}
\newtheorem*{claim*}{Claim}
\theoremstyle{definition}
\newtheorem*{definition}{Definition}
\newtheorem{remark}[theorem]{Remark}
\newtheorem{example}[theorem]{Example}
\newtheorem*{remark*}{Remark}

\newcommand{\N}{\mathbf{N}}

\newcommand{\Q}{\mathbf{Q}}

\newcommand{\F}{\mathbf{F}}

\renewcommand{\epsilon}{\varepsilon}

\DeclareMathOperator{\Sym}{Sym}

\DeclareMathOperator{\Hom}{Hom}
\DeclareMathOperator{\End}{End}

\DeclareMathOperator{\sgn}{sgn}

\renewcommand{\theta}{\vartheta}

\newcounter{thmlistcnt}
	{\setcounter{thmlistcnt}{0}%
	\begin{list}{\emph{(\roman{thmlistcnt})}}{%
		\usecounter{thmlistcnt}%
		\setlength{\topsep}{0pt}%
		\setlength{\leftmargin}{0pt}%
		\setlength{\itemsep}{0pt}%
		\setlength{\labelwidth}{17pt}
		\setlength{\itemindent}{30pt}}%
	}%
	{\end{list}}%

\newcommand{\size}[1]{\vert #1 \vert}

\usepackage{ytableau,tikz,graphicx,wrapfig,mathdots,mathtools}
\usepackage{enumerate}

\newcommand{\nat}{E^{\otimes r}}

\numberwithin{equation}{section}
\begin{document}
\title[]{Endomorphism algebras of 2-row permutation modules in characteristic~3}
\author{Jasdeep Kochhar}
\address{Department of Mathematics, Royal Holloway, University of London, United Kingdom}
\email{jasdeep.kochhar.2015@rhul.ac.uk}
\begin{abstract}
Given $r \in \mathbf{N},$ let $\lambda$ be a partition of $r$ with at most two parts.
Let $\mathbf{F}$ be a field of characteristic 3.  
Write $M^\lambda$ for the $\mathbf{F}S_r$-permutation module corresponding to the action of the symmetric group $S_r$ on the cosets of the maximal Young subgroup $S_\lambda.$ 
We construct a full set of central primitive idempotents in $\End_{\F S_r}(M^\lambda)$ in this case. 
We also determine the Young module corresponding to each primitive idempotent that we construct.
\end{abstract}
\date{\today}
\maketitle
\section{Introduction}
Given $r \in \N,$ let $S_r$ denote the symmetric group on $r$ letters.
Let $\lambda$ be a partition of $r,$ and write $S_\lambda$ for the Young subgroup of $S_r$ corresponding to $\lambda.$
Given a field $\F$, denote by $M^\lambda$ the $\F S_r$-permutation module corresponding to the action of $S_r$ on the cosets of $S_\lambda.$
The modules $M^\lambda$ are of central interest in the representation theory of the symmetric group. 
Over any field the Specht module $S^\lambda$ can be defined as a submodule of $M^\lambda.$
It is known that over the rational field the Specht modules are the irreducible $\Q S_n$-modules (see for instance \cite[\textsection 4]{J}). 

Our case of interest is when $\F$ is a field of positive characteristic. 
In this case James' Submodule Theorem \cite[Theorem 4.8]{J} implies that, up to isomorphism, there is a unique indecomposable summand of $M^\lambda$ containing $S^\lambda.$
We write $Y^\lambda$ for this summand, and we refer to this module as the \emph{Young module} labelled by $\lambda.$ 
Write $\unrhd$ for the dominance order of partitions. 
It is known (see \cite[Theorem 1]{E}) that $M^\lambda$ is in general isomorphic to a direct sum of Young modules $Y^\mu$ such that $\mu \unrhd \lambda.$
We can therefore write
\[M^\lambda \cong Y^\lambda \oplus \bigoplus_{\mu \rhd \lambda} [M^\lambda : Y^\mu] Y^\mu,\]
where $[M^\lambda : Y^\mu]$ denotes the number of indecomposable summands of $M^\lambda$ isomorphic to $Y^\mu.$
We refer to the multiplicity $[M^\lambda : Y^\mu]$ as a \emph{$p$-Kostka number.}
A complete characterisation of the $p$-Kostka numbers appears to be out of reach, however we mention a significant case in which they are known.
The modules in this case are those that we consider in this paper. 

Let $\lambda$ and $\mu$ be partitions of $r$ in at most two parts such that $\mu \unrhd \lambda.$
Write $\lambda = (\lambda_1,\lambda_2)$ and $\mu = (\mu_1,\mu_2).$ 
Define $m = \lambda_1-\lambda_2$ and $g = \lambda_2-\mu_2.$
Observe that $m \ge 0$ as $\lambda$ is a partition, and $g \ge 0$ as $\mu \unrhd \lambda.$
The main theorem in \cite{H} is that the $p$-Kostka number $[M^{(\lambda_1,\lambda_2)} : Y^{(\mu_1,\mu_2)}]$ is non-zero if and only if the binomial coefficient 
\[B(m,g) := \binom{m+2g}{g}\]
is non-zero modulo $p.$ 
This result is proved using a result of Donkin \cite[(3.6)]{DonkinTilting} based on Klyachko's multiplicity formula \cite[Corollary 9.2]{Klyachko}.
In the case that $Y^{(\mu_1,\mu_2)}$ is a summand of $M^{(\lambda_1,\lambda_2)},$ it also proved that the corresponding $p$-Kostka number equals one (see \cite[Lemma 3.2]{H}). 

In \cite{DEH} it is shown that the binomial coefficient $B(m,g)$ can also be used to construct the central primitive idempotents in $S_\F(\lambda):=\End_{\F S_r}(M^{(\lambda_1,\lambda_2)}),$ where $\F$ is a field of characteristic 2. 
The first main result in this paper is Theorem \ref{thm: main}, which constructs the central primitive idempotents in $S_\F(\lambda)$ when $\F$ is a field of characteristic 3. 
Our second main result is Theorem \ref{thm: correspondence}, which determines the Young modules that the primitive idempotents constructed in Theorem \ref{thm: main} correspond to.
This gives a construction of the Young modules $Y^{(\mu_1,\mu_2)}$ over a field of characteristic 3.

We remark that the proofs of our main theorems utilise various ideas from \cite{DEH}. 
Indeed in \cite{DEH} the authors give a presentation of $S_\F(\lambda)$ for any field $\F.$ 
We use the basis and corresponding multiplication formula in this presentation to construct the primitive idempotents in our case. 
Our construction of the primitive idempotents in $S_\F(\lambda)$ uses the same idea as \cite{DEH} of giving a correspondence between particular elements of $S_\F(\lambda)$ and the binomial coefficients $\binom{a}{b}$ such that $0 \le b \le a < p.$ 
The number of binomial coefficients of this form clearly increases with $p,$ and so it seems difficult to determine such a correspondence for fields of characteristic $p \ge 5.$ 
It is remarked in \cite[\S 1]{DEH} that explicitly constructing the primitive idempotents appears difficult even when $p = 3.$ 
As demonstrated by our main theorems, we completely solve the problem in this case. 
We also note that the argument used to prove that the idempotents we construct are primitive is based on the counting argument in \S 2.4 of \cite{DEH}.
Moreover, the proof of Theorem \ref{thm: correspondence} is taken directly from the proof of Theorem 7.1 in \cite{DEH}.
We repeat the proof here in order to make this paper more self-contained. 

We now describe where our ideas differ to those in \cite{DEH}. 
We will see in Lemma \ref{lem: mult} below that
the multiplication structure of $S_{\F}(\lambda)$ depends only on $m,$ 
whereas our construction of the primitive idempotents depends on $B(m,g).$
We are therefore required to determine the critical parameter $m$ given $m+2g$ and $g.$ 
An important observation in \cite{DEH} is that if $g$ has binary expansion $g = \sum_{i \ge 0} g_i2^i,$ then $2g$ has binary expansion $2g = \sum_{i \ge 1} g_{i-1}2^i.$ 
Furthermore, the proof of the Idempotent Theorem in \cite{DEH} uses that the sum of any two idempotents is an idempotent over a field of characteristic 2.
These two observations only hold when $p = 2,$ and so we take a different approach when proving the analogous results in our case (see \S \ref{sec: carries} and \S \ref{sec: proof1}).

In order to state our main theorems, we require background on the Schur algebra, which we give in \textsection \ref{sec: maintheorems}.
For details on the various connections between the representation theories of the symmetric group and the general linear group via the Schur algebra, we refer the reader to \cite{Green} and \cite{Martin}.  
\subsection{The Schur algebra}\label{sec: maintheorems}
Given $n, r \in \N,$ fix an $n$-dimensional $\F$-vector space $E$ with basis $\{v_1,\ldots,v_n\}.$ 
Form the $r$-fold tensor product $\nat,$ on which the symmetric group $S_r$ acts by place permutation. 
Extend this action linearly to the group algebra $\F S_r,$ and then define the \emph{Schur algebra} 
\[S_\F(n,r) = \End_{\F S_r}(\nat).\]
Given a partition $\lambda$ of $r,$ we realise the permutation module $M^{\lambda}$ as an $\F S_r$-submodule of $\nat.$
Define 
\[I(n,r) = \{(i_1,\ldots,i_r) : i_j \in \{1,2,\ldots,n\} \mbox{ for all $j$}\}.\]
We say that $(i_1,\ldots,i_r) \in I(n,r)$ has \emph{weight $\lambda$} if 
\[\size{\{j : i_j = k\}} = \lambda_k,\]
for all $1 \le k \le \ell(\lambda).$ 
For instance, the elements in $I(2,3)$ of weight $(2,1)$ are 
\[(1,1,2),(1,2,1) \mbox{ and } (2,1,1).\]
Then $M^\lambda$ is isomorphic to the $\F$-span of the set
\[\{v_{i_1} \otimes \cdots \otimes v_{i_r} : (i_1,i_2,\ldots,i_r) \mbox{ has weight $\lambda$}\}.\]
We remark that there is the analogous construction of $M^\lambda$ when $\lambda$ is a composition of $r.$
It is then easy to see that for a composition $\lambda$ of $r,$ there is an isomorphism of $\F S_r$-modules $M^\lambda \cong M^{\overline\lambda},$ where $\overline{\lambda}$ is the partition of $r$ obtained by writing the parts of $\lambda$ in non-increasing order. 
Then there is a decomposition of $\F S_r$-modules
\[\nat = \bigoplus_{\lambda \in \Lambda(n,r)} M^{\lambda},\]
where $\Lambda(n,r)$ denotes the set of compositions of $r$ with at most $n$ parts. 

We are interested in partitions of $r$ with at most two parts, and so we fix $n = 2$ throughout the rest of the paper.  
The main result in \cite{DG} is an explicit presentation of $S_\Q(2,r)$ as a quotient of the universal enveloping algebra $U(\mathfrak{gl}_2).$ 
This result can be used to give an explicit presentation of the endomorphism algebra $S_\F(\lambda),$ which we now describe. 
Following the notation in \cite{DEH} and \cite{DG}, define $e = e_{21}, f = e_{12}, H_1 = e_{11},$ and $H_2 = e_{22},$ where $e_{ij}$ is the standard matrix unit in $\mathfrak{gl}_2.$  
As in \cite[3.4]{DG}, given $\ell \in \N_0$ and an element $T$ in an associative $\Q$-algebra with 1, we define 
\[T^{(\ell)} = \frac{T^\ell}{\ell!} \mbox{ and } \binom{T}{\ell} = \dfrac{T(T-1)\ldots(T-\ell+1)}{\ell!}.\]
Then given $\lambda = (\lambda_1,\lambda_2) \in \Lambda(2,r),$ define
\[1_\lambda = \binom{H_1}{\lambda_1}\binom{H_2}{\lambda_2}.\]
It is proved in \cite[Lemma 5.3]{DEH2} that $1_\lambda$ is an idempotent in $S_\Q(2,r),$ and that $1_\lambda E^{\otimes r} = M^{\lambda}.$
Given $i \in \N_0,$ we define 
\[b(i) = 1_\lambda f^{(i)}e^{(i)} 1_\lambda.\]

The following lemma completely describes $S_\F(\lambda)$ as an associative $\F$-algebra. 
We remark that this lemma is an equivalent restatement of Proposition 3.6 in \cite{DEH}, chosen to make it obvious that $S_\F(\lambda)$ is commutative. 
\begin{lemma}[{\cite[Proposition 3.6]{DEH}}]\label{lem: mult}
Given $r \in \N,$ let $\lambda = (\lambda_1,\lambda_2) \vdash r,$ and define $m = \lambda_1-\lambda_2.$  
Then $S_\F(\lambda) $ has an $\F$-basis given by the set 
\[\{ b(i) : 0 \le i \le \lambda_2\}.\]
Moreover, the multiplication of the basis elements is given by the formula
\[b(i)b(j) = \sum_{h = \max\{i,j\}}^{i+j} \binom{h}{i}\binom{h}{j}\binom{m+i+j}{i+j-h} b(h),\]
where we set $b(a) = 0$ if $a > \lambda_2.$ 
\end{lemma} 
We refer to the basis given in this lemma as the \emph{canonical basis} of $S_\F(\lambda).$ 
The presentation of the Schur algebra in \cite{DG} is over the field $\Q.$
Nevertheless $b(i)$ is well-defined over a field of characteristic $p.$
Moreover, the structure constants given in Lemma \ref{lem: mult} are integers. 
Therefore the above multiplication formula holds over a field of characteristic $p$ by reducing the coefficients modulo $p.$ 
Furthermore, the $\Q S_r$-module $M^\lambda$ is multiplicity free, and so $S_\Q(\lambda)$ is a commutative algebra.
This implies that $S_\F(\lambda)$ is also a commutative algebra. 
Also a direct computation using the multiplication formula shows that $b(0)$ is the identity in $S_\F(\lambda),$
and we write $\mathbf{1}$ for $b(0).$ 

Throughout the rest of this section, we assume that $\F$ is a field of characteristic 3. 
We now define the elements $e_{m,g} \in S_\F(\lambda),$ which are the subject of Theorem \ref{thm: main} (see below).
Let $m,g \in \N_0$ be such that $B(m,g)$ is non-zero modulo $3.$ 
Let $m+2g = \sum_{u \ge 0} (m+2g)_u 3^u$ and $g = \sum_{u \ge 0} g_u 3^u$ be the 3-adic expansions of $m+2g$ and $g,$ respectively. 
Define the index sets
\begin{align*}
I^{(0)}_{m,g} &= \{u : g_u = 0\mbox{ and }(m+2g)_u = 0\}\\ 
J^{(0)}_{m,g} &= \{u : g_u = 1\mbox{ and }(m+2g)_u = 2\}\\
I^{(1)}_{m,g} &= \{u : g_u = 0\mbox{ and }(m+2g)_u = 1\}\\ 
J^{(1)}_{m,g} &= \{u : g_u = 2\mbox{ and }(m+2g)_u = 2\}\\
I^{(2)}_{m,g} &= \{u : g_u = 0\mbox{ and }(m+2g)_u = 2\}\\ 
J^{(2)}_{m,g} &= \{u : g_u = 1\mbox{ and }(m+2g)_u = 1\}.
\end{align*}
The chosen notation for these index sets may not seem intuitive upon first reading, 
but the results in \textsection \ref{sec: carries} will make this clear. 

Define 
\begin{align*}
e_{m,g} &= \prod_{u \in I^{(0)}_{m,g}} \left(\mathbf{1} + b(3^u)- b(2\cdot 3^u)\right) \cdot \prod_{u \in J^{(0)}_{m,g}} \left(b(2\cdot 3^u) - b(3^u)\right)\\
&\cdot \prod_{u \in I^{(1)}_{m,g}} \left(\mathbf{1} - b(2\cdot 3^u)\right)\cdot \prod_{u \in J^{(1)}_{m,g}} b(2\cdot 3^u) \\
&\cdot \prod_{u \in I^{(2)}_{m,g}} \left(\mathbf{1} - b(3^u) + b(2\cdot 3^u)\right) \cdot \prod_{u \in J^{(2)}_{m,g}} \left(b(3^u) - b(2\cdot 3^u)\right).
\end{align*}
As stated in Lemma \ref{lem: mult}, if $b(a)$ in this product is such that $a > \lambda_2,$ then we set $b(a) = 0.$
Given $t \in \N_0,$ we define $(e_{m,g})_{\le t}$ by taking the products defining $e_{m,g}$ over the $u$ in each index set such that $u \le t,$ 
and we define $(e_{m,g})_{< t}$ in the analogous way. 
We give an example of $e_{m,g}$ in \textsection \ref{sec: notation}. 

We are now ready to state our first main theorem. 
\begin{theorem}\label{thm: main} 
Given $r \in \N,$ let $\lambda = (\lambda_1,\lambda_2) \vdash r$ and $m = \lambda_1-\lambda_2.$
The set of elements $e_{m,g}$, with $B(m,g)$ non-zero modulo $3$ and $g \le \lambda_2,$ is a complete set of primitive orthogonal idempotents for $S_\F(\lambda)$. 
\end{theorem}
Theorem \ref{thm: main} implies that $e_{m,g}M^\lambda \cong Y^{\mu}$ for some $\mu = (\mu_1,\mu_2) \vdash r$ such that $\mu \unrhd \lambda.$
Our second main theorem determines $\mu$ in this case. 
\begin{theorem}\label{thm: correspondence}
Let $\lambda = (\lambda_1, \lambda_2)$ and $\mu = (\mu_1, \mu_2)$ be partitions of $r$ such that $Y^\mu$ is a direct summand of $M^\lambda$. 
Define 
\[m= \lambda_1-\lambda_2 \mbox{ and } g=\lambda_2 - \mu_2.\] 
Then $e_{m,g}$ is the primitive idempotent in $S_\F(\lambda)$ such that $e_{m,g}M^\lambda \cong Y^\mu.$
\end{theorem} 
\subsection{Outline}\label{sec: outline}  
In \textsection \ref{sec: notation} we give the notation used throughout the paper.
We highlight that we define the $p$-adic expansion of a binomial coefficient $\binom{a}{b},$ as $e_{m,g}$ is implicitly constructed using the 3-adic expansion of $B(m,g).$ 

In \textsection \ref{sec: squares} we consider more closely the multiplication structure of $S_\F(\lambda).$
In particular, we define the element $\psi_{m,u},$ where $u \in \N_0.$
The product of $(e_{m,g})_{< u}$ (defined on the previous page) and $\psi_{m,u}$ is fundamental in the proof of Theorem \ref{thm: main}. 

As can be observed in Lemma \ref{lem: mult}, the critical parameter in the multiplication formula for $S_\F(\lambda)$ is $m.$ 
In \textsection \ref{sec: carries} we therefore relate the 3-adic expansion of $B(m,g)$ to the 3-adic expansion of $m.$
We see that this depends on the carries in the ternary addition of $m$ and $g.$ 

In \textsection \ref{sec: proof1} we prove Theorem \ref{thm: main}. 
We prove Proposition \ref{prop: emgu}, which states that the elements $(e_{m,g})_{\le u}$ are idempotents for all $u \in \N_0.$
Before we prove Proposition \ref{prop: emgu}, we show how it implies that the elements $e_{m,g}$ are idempotents in $S_\F(\lambda).$ 
The proof of Proposition \ref{prop: emgu} is by induction on $u.$ 
We give the base case in \textsection \ref{sec: base}, and we complete the inductive step in \textsection \ref{sec: inductivestep}.
In \textsection \ref{sec: primitive} we show that the elements $e_{m,g}$ are mutually orthogonal.
A simple counting argument then shows that these elements give a complete set of primitive orthogonal idempotents in $S_\F(\lambda),$ which completes the proof of Theorem \ref{thm: main}.   

In \textsection \ref{sec: proof2} we prove Theorem \ref{thm: correspondence}.
Following the exposition in \cite{DEH}, the proof by induction on $r.$
Observe that $m$ and $g$ are invariant under adding the partition $(1^2)$ to both $\lambda$ and $\mu$.
In the inductive step we therefore prove that if $e_{m,g}M^\lambda \cong Y^\mu,$ then $e_{m,g}M^{\lambda+(1^2)} \cong Y^{\mu + (1^2)}.$
We remark that this is an algebraic realisation of the column removal phenomenon for the decomposition matrices of symmetric groups proved by James (see \cite{JamesColRemoval}). 
\section{Notation}\label{sec: notation}
Let $p$ be a prime number.
Given $a \in \N_0$ with $p$-adic expansion $a = \sum_{u = 0}^t a_up^u,$ we write $a =_p [a_0,a_1,\ldots,a_t].$ 
Given $s \in \N,$ we write $a_{< s}$ for $\sum_{u = 0}^{s-1} a_up^u.$ 
Also given $b =_p [b_0,b_1,\ldots,b_t],$ Lucas' Theorem states that 
\[\binom{a}{b} \equiv_p \prod_{u = 0}^t \binom{a_u}{b_u}.\]
We refer to the factorisation on the right hand side as the \emph{$p$-adic expansion} of $\binom{a}{b}.$
Define \emph{factor u} in the $p$-adic expansion of $\binom{a}{b}$ as the binomial coefficient $\binom{a_u}{b_u}.$
Given $m,g \in \N_0,$ we write $B(m,g)_p$ for the $p$-adic expansion of $B(m,g).$ 

Recall from Lemma \ref{lem: mult} that $S_\F(\lambda)$ has an $\F$-basis equal to
\[\{ b(i) : 0 \le i \le \lambda_2\},\]
and that $\mathbf{1}$ denotes $b(0) = 1_{S_\F(\lambda)}.$ 
We also define the order $\le$ on the $b(i)$ by $b(i) \le b(j)$ if and only if $i \le j.$ 

We remark that we can define $e_{m,g}$ by assigning elements in $S_\F(\lambda)$ to all possible factors of $B(m,g)_3$, and then multiplying these elements of $S_\F(\lambda)$ according to the factors of $B(m,g)_3$ (see Example \ref{ex: main} below).
The assignment to factor $u$ of $B(m,g)_3$ is as follows: 
\begin{align*}
\binom{0}{0} & \leftrightarrow  \mathbf1+b(3^u)-b(2\cdot3^u) & \binom{2}{1} & \leftrightarrow b(2\cdot3^u) - b(3^u) \\
\binom{1}{0} & \leftrightarrow \mathbf1-b(2\cdot 3^u)& \binom{2}{2} & \leftrightarrow b(2\cdot 3^u)\\
\binom{2}{0} & \leftrightarrow \mathbf1-b(3^u)+b(2\cdot 3^u)& \binom{1}{1} & \leftrightarrow b(3^u) - b(2\cdot 3^u),
\end{align*}
and assigning zero to any other factor of $B(m,g)_3.$
Observe that if $B(m,g) = 0,$ then $e_{m,g} = 0$ according to this definition.
We define \emph{factor u} of $e_{m,g}$ as the factor of $e_{m,g}$ corresponding to factor $u$ of $B(m,g)_3.$  

We give an example of $e_{m,g}$ below.
Before we do this, we state the following useful lemma from \cite{DEH}.
\begin{lemma}[{\cite[Lemma 3.7]{DEH}}]\label{lem: bipadic}
Let $p$ be a prime number, and let $i \in \N$ be such that $i =_p [i_0,i_1,\ldots].$ 
Then $b(i)= \prod_{t \ge 0} b(i_t\cdot p^t).$
\end{lemma}
\begin{example}\label{ex: main}
Let $\lambda = (36,13),$ and let $\mu = (49,0).$ 
Then $m = 23,$ $g = 13,$ and
\[B(23,13)_3= \binom{1}{1}\binom{1}{1}\binom{2}{1}\binom{1}{0}\binom{0}{0}\binom{0}{0}\ldots.\]
Therefore $e_{23,13}$ equals
\[(b(1) - b(2))(b(3)-b(6))(b(18)-b(9))(\mathbf1-b(54))(\mathbf1+b(81)-b(162))\ldots.\]
As $b(a) = 0$ for $a > 13$ in $S_\F((36,13)),$ only finitely many factors in this infinite product are not equal to $\mathbf{1}.$ 
Then by Lemma \ref{lem: bipadic}
\begin{align*}
e_{23,13} &= (b(1) - b(2))(b(3)-b(6))(-b(9))\\
&= -b(13) + b(14) + b(16) - b(17)\\ 
&= -b(13) 
\end{align*}
in $S_\F((36,13)).$
\end{example}

\section{Multiplication in $S_\F(\lambda)$}\label{sec: squares}
Throughout this section fix $m \in \N_0,$ and fix a partition $\lambda = (\lambda_1,\lambda_2)$ such that $m=\lambda_1-\lambda_2.$ 
Observe that factor $u$ of $e_{m,g}$ can be expressed in terms of the elements
\begin{equation}\label{eq: factorsquares}
b(2\cdot 3^u) - b(3^u) \mbox{ and } b(2\cdot 3^u),
\end{equation} 
where $u \in \N_0.$ 
In the proof of Theorem \ref{thm: main}, we show that $(e_{m,g})_{\le u}^2 = (e_{m,g})_{\le u}.$
To this end we need to determine the squares of the elements in \eqref{eq: factorsquares}.
In this section we therefore assume that $\lambda_2 \ge 2\cdot 3^u,$ and we consider the products $b(3^u)^2, b(2\cdot 3^u)^2,$ and $b(3^u)b(2\cdot 3^u).$
\begin{definition}
Given $u \in \N_0,$ define
\[\psi_{m,u} = \sum_{k = 1}^{3^u - 1}\binom{m_{< u}}{3^u - k}b(k).\]
\end{definition}
We remark that our motivation for defining $\psi_{m,u}$ is twofold.
The immediate reason is that we can express the products $b(3^u)^2, b(2\cdot 3^u)^2,$ and $b(3^u)b(2\cdot 3^u)$ in terms of $\psi_{m,u}.$
Also, as stated in the outline, the product of $\psi_{m,u}$ with $(e_{m,g})_{< u}$ is fundamental in the proof of Theorem \ref{thm: main}. 

Consider first $b(3^u)^2.$ 
Lemma \ref{lem: mult} gives
\[b(3^u)^2 = \sum_{h = 3^u}^{2\cdot3^u} \binom{h}{3^u}^{\hspace{-3pt} 2}\binom{m + 2\cdot 3^u}{2\cdot 3^u - h}b(h).\]
A direct computation using this formula shows that the coefficient of $b(3^u)$ in $b(3^u)^2$ equals $\binom{m_u+2}{1},$ and that the coefficient of $b(2\cdot 3^u)$ equals 1. 
Also observe that in this sum if $3^u < h < 2\cdot 3^u,$ then we can write $h = 3^u + k,$ where $0 < k < 3^u.$
Then by Lucas' Theorem, for all such $h$
\[\binom{m+ 2\cdot 3^u }{2\cdot 3^u - h} \equiv_3 \binom{m_{< u}}{3^u-k} \binom{m_u+2}{0}  \equiv_3 \binom{m_{< u}}{3^u-k},\]
and so using Lemma \ref{lem: bipadic} we can write
\begin{equation}\label{eq: b3sq}
b(3^u)^2 = b(3^u)\left[\binom{m_u+2}{1} + \psi_{m,u}\right]+b(2\cdot 3^u).
\end{equation}

Consider now
\[b(2\cdot 3^u)^2 = \sum_{h = 2\cdot 3^u}^{4\cdot 3^u} \binom{h}{2\cdot 3^u}^{\hspace{-3pt} 2}\binom{m + 3^u + 3^{u+1}}{4\cdot 3^u - h}b(h).\]
Observe that if $h \ge 3^{u+1}$ in this sum, then the ternary addition of $2\cdot 3^u$ and $h - 2\cdot 3^u$ is not carry free.
It follows that  $\binom{h}{2\cdot 3^u} \equiv_3 0.$ 
Arguing similarly as above, the coefficient of $b(2\cdot 3^u)$ in $b(2\cdot 3^u)^2$ equals $\binom{m_u+1}{2}$.
Moreover, if $2\cdot 3^u < k < 3^{u+1},$ then we can write $h = 2\cdot 3^u + k,$ where $0 < k < 3^u.$
Again by Lucas' Theorem, for all such $h$ 
\[\binom{m + 3^u + 3^{u+1}}{4\cdot 3^u - h} = \binom{m + 3^u + 3^{u+1}}{3^u + 3^u - k} \equiv_3 \binom{m_{< u}}{3^u-k}\binom{m_u+ 1}{1}.\]
Using Lemma \ref{lem: bipadic} once more we obtain
\begin{equation}\label{eq: b2sq}
b(2\cdot 3^u)^2 = b(2\cdot 3^u)\left[\binom{m_u+1}{2} + \binom{m_u+1}{1}\psi_{m,u}\right].
\end{equation}
An entirely similar argument gives
\begin{equation}\label{eq: b3b2}
b(3^u)b(2\cdot 3^u) = b(2\cdot 3^u)\left[2\binom{m_u}{1} - \psi_{m,u}\right].
\end{equation}
If $j$ is maximal such that $b(j)$ appears with non-zero coefficient in one of $b(3^u)^2, b(3^u)b(2\cdot 3^u),$ or $b(2\cdot 3^u)^2,$ then \eqref{eq: b3sq}, \eqref{eq: b2sq} and \eqref{eq: b3b2} show that $j < 3^{u+1}.$
We therefore have the following lemma, which will be used in the inductive step of the proof of Proposition \ref{prop: emgu}. 
\begin{lemma}\label{lem: highestdegree}
Let $u \in \N$ be such that $3^u \le \lambda_2.$
Then the $\F$-span of the set 
\[\{b(k) : k < 3^u\}\]
is a subalgebra of $S_\F(\lambda).$ 
\end{lemma}
We end this section with the following lemma, which determines when $e_{m,g}$ is non-zero in $S_\F(\lambda).$ 
We remark that the first statement of the lemma can be observed in Example \ref{ex: main}. 
\begin{lemma}\label{lem: smallestinemg}
Let $g \in \N_0$ be such that $B(m,g)$ is non-zero modulo $3.$ 
Then 
\[e_{m,g} = B(m,g)b(g) + \sum_{i > g} \alpha_ib(i),\]
for some $\alpha_i \in \F_3.$ 
In particular, $e_{m,g}$ is non-zero in $S_\F(\lambda)$ if and only if $g \le \lambda_2.$ 
\end{lemma}
\begin{proof}
Write $e_{m,g}$ as a linear combination in the canonical basis of $S_\F(\lambda)$ given in Lemma \ref{lem: mult}.  
As the index sets defining $e_{m,g}$ are mutually disjoint, Lemma \ref{lem: bipadic} implies that the smallest term in $e_{m,g}$ is the product of the smallest term in each factor (see \S \ref{sec: notation}) of $e_{m,g}.$
By the construction of $e_{m,g}$ immediately before Lemma \ref{lem: bipadic}, the smallest term in factor $u$ of $e_{m,g}$ is $b(g_u3^u)$ with coefficient $\binom{(m+2g)_u}{g_u}.$ 
It follows that the smallest term in $e_{m,g}$ is $\prod_u b(g_u3^u) = b(g)$ with coefficient $\prod_u \binom{(m+2g)_u}{g_u} \equiv_3 B(m,g).$

The second statement of the lemma now follows, since the largest element in the canonical basis of $S_\F(\lambda)$ is $b(\lambda_2).$ 
\end{proof}

\section{Analysis of the binomial coefficient $B(m,g)$}\label{sec: carries}
Fix a prime number $p,$ and let $m,g \in \N_0$ such that $B(m,g)$ is non-zero modulo $p.$
In this section we use the $p$-adic expansion of $B(m,g)$ to understand $m.$ 
We see that we can do this using the $p$-ary addition of $m$ and $g.$
We begin by considering the Example \ref{example: main} below, which demonstrates the link between $B(m,g)$ and $m$ that occurs in the general case.  
We require the following notation.

Given a prime $p,$ consider the following representation of the $p$-ary addition of $m$ and $g$:
\[
\begin{array}{r | cccccccccc}
m & m_0 & m_1 & \ldots & m_u & \ldots\\
g & g_0 & g_1 & \ldots & g_u & \ldots\\\hline
m+g & (m+g)_0 & (m+g)_1 & \ldots & (m+g)_u & \ldots
\end{array}\ ,
\]
where $m =_p [m_0,m_1,\ldots],$ and the analogous statements hold for $g$ and $m+g.$
Define $x_{-1} = 0,$ and given $u \in \N_0,$ recursively define $x_u$ as follows:
\begin{equation}\label{eq: carry}
m_u + g_u + x_{u-1} = (m+g)_u + px_u,
\end{equation}
so that $x_u$ is the carry \emph{leaving} column $u$ in this addition. 
Therefore for all $u \in \N_0,$ $x_{u-1}$ is the carry \emph{entering} column $u$ in this addition. 

\begin{remark}
The carries $x_u$ serve two purposes in this paper. 
The first, as we will see in this section, is that we can determine $m_u$ using $x_{u-1}.$
The second is that the product $(e_{m,g})_{< u}\psi_{m,u}$
can be determined entirely by the carry $x_{u-1}$
(see Lemma \ref{lem: finalfactor} below).  
We admit that it remains mysterious to us as to why 
this product depends only on $x_{u-1}.$ 
\end{remark}
In the next lemma, we determine the possible values of the carry $x_u.$
\begin{lemma}\label{lem: carry}
Suppose that in the $p$-ary addition of $m$ and $g$ the carry $x_u$ is non-zero for some $u \in \N_0.$
Then $x_u = 1.$ 
\end{lemma}
\begin{proof}
Fix $u \in \N_0.$ 
By definition of $x_u,$ we have that 
\[\sum_{j=0}^u m_j p^j + \sum_{j=0}^u g_j p^j = x_u p^{u+1} + \sum_{j=0}^u (m+g)_j p^j.\]
Each of the sums on the left hand side of this equation is strictly less than $p^{u+1},$ and so 
\[2p^{u+1} > x_u p^{u+1} + \sum_{j=0}^u (m+g)_j p^j.\]
The result now follows since the sum on the right hand side of the inequality is non-negative. 
%
\end{proof} 
We now use Lemma \ref{lem: carry} to determine the possibilities for $m_u$ given factor $u$ of $B(m,g)_p.$
\begin{lemma}\label{lem: coltom}
Let $m,g \in \N_0$ be such that $B(m,g)$ is non-zero modulo $p.$ 
Let $a,b \in \N_0$ be such that $0 \le b \le a <p,$ and let factor $u$ of $B(m,g)_p$ equal $\binom{a}{b}.$
Let $z \in \{0,1,\ldots,p-1\}$ be the unique integer such that $z \equiv_p a-2b.$
Then either $m_u \equiv_p z$ and $x_{u-1} = 0,$ or $m_u \equiv_p z-1$ and $x_{u-1} = 1.$  
Moreover, $x_u = 1$ if and only if $m_u + g_u + x_{u-1} \ge p.$ 
\end{lemma}
\begin{proof}
It follows from the definition of $B(m,g)_p$ that $(m+2g)_u = a$ and $g_u = b.$
As $B(m,g)$ is non-zero modulo $p,$ it follows that the $p$-ary addition of $m+g$ and $g$ is carry free.
Therefore $(m+g)_u = a-b,$ and so 
\[m_u + b + x_{u-1} = a-b + px_u \equiv_p a-b.\]
By Lemma \ref{lem: carry}, we have that $0 \le x_{u-1} \le 1.$
If $x_{u-1} = 0,$ then $m_u \equiv_p a-2b = z.$
Similarly if $x_{u-1} = 1,$ then $m_u \equiv_p z-1,$ as required. 

The second statement is immediate by definition of the carry $x_u$ and Lemma \ref{lem: carry}.
\end{proof}
In particular Lemma \ref{lem: coltom} shows that $(m+2g)_u - 2g_u \equiv_p m_u + x_{u-1}$ for all $u \in \N_0$ whenever $B(m,g)$ is non-zero modulo $p.$  
We now give an example of this observation. 
\begin{example}\label{example: main}
Let $\mu \in \N_0$ and $\nu \in \N$ be such that $\nu > \mu.$
Let $h \in \N$ be such that $h < p^\mu$ and $\binom{2h}{h}$ is non-zero modulo $p.$

We consider the case when $m = p^\mu$ and $g = p^\nu - p^\mu + h.$
Then 
\begin{itemize}
\item $x_u = 0$ for $0 \le u \le \mu-1,$ 
\item and $x_u = 1$ for $\mu \le u \le \nu-1.$ 
\end{itemize}

Let $h_u$ be the digits in the $p$-adic expansion of $h.$ 
The conditions on $h$ imply that $h_u \le \frac{p-1}{2}$ for all $u,$ and $h_u = 0$ for $u \ge \mu.$ 
Then $m+2g = p^\nu + (p^\nu - p^\mu) + 2h,$ and so the $p$-adic expansion of $\binom{m+2g}{g}$ equals
\[\binom{2h_0}{h_0}\binom{2h_1}{h_1}\ldots\binom{2h_{\mu-1}}{h_{\mu-1}}\binom{p-1}{p-1}\ldots\binom{p-1}{p-1}\binom{1}{0},\]
where the rightmost factor appearing is factor $\nu.$ 

Observe that
\begin{itemize}
\item if $u < \mu,$ then $(m+2g)_u-2g_u = 0 = m_u,$
\item $(m+2g)_\mu-2g_\mu \equiv_p 1 = m_\mu,$
\item and if $\mu< u \le \nu,$ then $(m+2g)_u-2g_u \equiv_p 1 = m_u + 1.$
\end{itemize}
In all cases we can therefore write 
\[(m+2g)_u - 2g_u \equiv_p m_u + x_{u-1},\]
as expected from Lemma \ref{lem: coltom}. 
\end{example}

\section{Proof of Theorem \ref{thm: main}.}\label{sec: proof1}
Fix $m,g \in \N_0$ such that $B(m,g)$ is non-zero modulo 3, and let $\lambda = (\lambda_1,\lambda_2)$ be such that $m  = \lambda_1-\lambda_2.$ 
Throughout the rest of this paper, $\F$ is assumed to be a field of characteristic 3. 
We prove the following proposition by filling in the details in the outline in \textsection \ref{sec: outline}.
\begin{proposition}\label{prop: emgu}
Given $u \in \N_0,$ $(e_{m,g})_{\le u}$ is an idempotent in $S_\F((m+3^{u+1}-1,3^{u+1}-1)).$ 
\end{proposition}
We remark that Proposition \ref{prop: emgu}, together with Lemma \ref{lem: highestdegree}, implies that $(e_{m,g})_{\le u}$ is also idempotent in $S_\F((m+a,a))$ for all $a \ge 3^{u+1}.$ 

We prove Proposition \ref{prop: emgu} by induction on $u,$ in which the base case is $u = 0.$
Before we do this, we show how the proposition implies that $e_{m,g}$ is an idempotent in $S_\F(\lambda).$
By Lemma \ref{lem: mult}, $S_\F(\lambda)$ has a basis given by the set
\[\{b(i) : 0 \le i \le \lambda_2\}.\]
Let $u \in \N_0$ be such that $3^u \le \lambda_2 < 3^{u+1}.$
If $e_{m,g}$ is non-zero in $S_\F(\lambda),$ then Lemma \ref{lem: smallestinemg} gives that $g \le \lambda_2.$
Therefore $g < 3^{u+1},$ and so by construction, $(e_{m,g})_{\le u} = e_{m,g}$ when viewed as an element of $S_\F(\lambda).$ 
As the multiplication structure of $S_\F(\lambda)$ depends only on $m,$ Proposition \ref{prop: emgu} gives
\begin{align*}
(e_{m,g})^2 = ((e_{m,g})_{\le u})^2 = (e_{m,g})_{\le u} = e_{m,g} \in S_\F(\lambda),
\end{align*}
as required

We now proceed with the proof of Proposition \ref{prop: emgu}. 
\subsection{The base case}\label{sec: base}
By definition $x_{-1} = 0.$ 
In this case Lemma \ref{lem: coltom} states that factor $0$ of $B(m,g)_3$ equals $\binom{a}{b},$ where
$a - 2b \equiv_3 m_0.$
We distinguish three cases, determined by $m_0.$

\emph{Case} (1). Suppose that $m_0 = 0.$ 
Then the only possibilities for factor 0 of $B(m,g)_3$ are 
\[\binom{0}{0} \mbox{ or } \binom{2}{1}.\]
By definition $(e_{m,g})_{\le 0}$ equals either $\mathbf{1}-b(1)+b(2)$ or $b(2) - b(1).$ 
It is sufficient to prove that $b(2) - b(1)$ is idempotent when $m_0 = 0.$
Indeed \eqref{eq: b3sq}, \eqref{eq: b2sq} and \eqref{eq: b3b2} applied with $u = 0$ and $m_0 = 0$ give
\begin{align*}
(b(2)-b(1))^2 &= b(2)^2 + b(1)b(2) + b(1)^2\\
&= 0 + 0 +b(2) - b(1) = b(2) - b(1).
\end{align*}

\emph{Case} (2). Suppose that $m_0 = 1.$
Then the only possibilities for factor 0 of $B(m,g)_3$ are 
\[\binom{1}{0} \mbox{ or } \binom{2}{2},\]
and so $(e_{m,g})_{\le 0}$ equals either $\mathbf{1}-b(2)$ or $b(2).$
Applying \eqref{eq: b2sq} with $u = 0$ and $m_0 = 1$ shows that $b(2)$ is idempotent in this case.

\emph{Case} (3). Suppose that $m_0 = 2.$
Then the only possibilities for factor 0 of $B(m,g)_3$ are 
\[\binom{2}{0} \mbox{ or } \binom{1}{1},\]
and so $(e_{m,g})_{\le 0}$ equals either $\mathbf{1}-b(1)+b(2)$ or $b(1) - b(2).$ 
Again \eqref{eq: b3sq}, \eqref{eq: b2sq} and \eqref{eq: b3b2} applied with $u = 0$ and $m_0 = 2$ give
\begin{align*}
(b(1) - b(2))^2 &= b(1)^2 + b(1)b(2) + b(2)^2\\
 &= b(1) + b(2) +b(2) + 0 \equiv_3 b(1) -b(2),
\end{align*}
as required. 
\subsection{The inductive step}\label{sec: inductivestep}
Throughout this section fix $u \in \N.$ 
Lemma \ref{lem: highestdegree} implies that $((e_{m,g})_{\le u})^2$ is contained in the $\F$-span of $\{b(i) : i < 3^{u+1}\},$
and so it is sufficient to prove that $(e_{m,g})_{\le u}$ is an idempotent in $S_\F((m+\lambda_2,\lambda_2)),$ where $\lambda_2 < 3^{u+1}.$ 

Assume inductively that $(e_{m,g})_{\le t}$ is an idempotent in $S_\F(\lambda)$ for all $t < u.$
We require the following lemmas. 
\begin{lemma}\label{lem: multonend}
Let $t \in \N_0$ be such that $t < u.$
Suppose that $v:=(e_{m,g})_{\le t}w,$ is an idempotent in $S_F(\lambda).$
Then $vw=v$ and $v(\mathbf{1}-w)=0.$
\end{lemma}
\begin{proof}
We have assumed that $(e_{m,g})_{\le t}$ is an idempotent in $S_F(\lambda),$ and so 
\[vw = (e_{m,g})_{\le t}w^2 =  ((e_{m,g})_{\le t})^2w^2 = v^2 = v,\]
as required. 
The proof that $v(\mathbf{1}-w)=0$ is entirely similar. 
\end{proof}
Recall from \textsection \ref{sec: carries} that $x_t$ denotes the carry leaving column $t$ in the ternary addition of $m$ and $g,$ and that 
\[\psi_{m,t} =  \sum_{k = 1}^{3^t - 1}\binom{m_{< t}}{3^t - k}b(k)\]
for $t \in \N_0.$
\begin{lemma}\label{lem: finalfactor}
Let $t \in \N_0$ be such that $t \le u.$
Then 
\[(e_{m,g})_{< t }\psi_{m,t} = 
\begin{cases}
0 & \mbox{if $x_{t-1} = 0,$}\\
(e_{m,g})_{< t}  &  \mbox{if $x_{t-1} = 1.$}
\end{cases}\]
\end{lemma}
\begin{proof}
We proceed by induction on $t.$ 
The base case is when $t = 0,$ where the product defining $(e_{m,g})_{< 0}$ is empty.
Therefore $(e_{m,g})_{< 0} = 1.$
By definition $x_{-1}=0$ and $\psi_{m,0} = 0,$ and so the result holds in this case. 

Suppose now that $t \ge 1$ and that the result holds for all $s < t.$  
By Lemma \ref{lem: bipadic} we can write
\begin{align*}
\psi_{m,t} &= \sum_{k=1}^{3^{t-1}-1} \binom{m_{< t}}{3^t - k}b(k)\\ 
&+ b(3^{t-1})\left[\binom{m_{t-1}}{2}+\sum_{k=1}^{3^{t-1}-1}\binom{m_{< t}}{3^t - (3^{t-1}+k)}b(k)\right]\\
& + b(2\cdot 3^{t-1})\left[\binom{m_{t-1}}{1}+ \sum_{k=1}^{3^{t-1}-1}\binom{m_{< t}}{3^t - (2\cdot 3^{t-1}+k)}b(k)\right].
\end{align*}
For $1 \le k \le 3^{t-1}-1,$ Lucas' Theorem implies that
\begin{align*}
\binom{m_{< t}}{3^t - k} &= \binom{m_{< t-1} + m_{t-1}\cdot 3^{t-1}}{3^{t-1}-k + 2\cdot 3^{t-1}}\\
&\equiv_3 \binom{m_{< t-1}}{3^{t-1}-k}\binom{m_{t-1}}{2}.
\end{align*}
Applying similar arguments for all $3^{t-1} \le k \le 3^t-1$ shows that
\begin{align}
\psi_{m,t} &= \psi_{m,t-1}\left[\binom{m_{t-1}}{2} + \binom{m_{t-1}}{1}b(3^{t-1}) + \binom{m_{t-1}}{0}b(2\cdot 3^{t-1})\right]\label{eq: psimtrecursion}\\
&+\binom{m_{t-1}}{2}b(3^{t-1}) + \binom{m_{t-1}}{1}b(2\cdot 3^{t-1}).\nonumber
\end{align}
We now distinguish three cases, determined by $m_{t-1}.$

{\it Case} (1). Suppose that $m_{t-1} = 0.$ 
Then \eqref{eq: psimtrecursion} becomes
\[\psi_{m,t} = \psi_{m,t-1}b(2\cdot 3^{t-1}).\]

If $x_{t-2} = 0,$ then the first statement of Lemma \ref{lem: coltom} implies that factor $t-1$ of $B(m,g)_3$ equals either $\binom{0}{0}$ or $\binom{2}{1}.$ 
As $x_{t-2} = m_{t-1} = 0,$ the second statement of Lemma \ref{lem: coltom} gives that $x_{t-1} = 0.$
Moreover, the inductive hypothesis of this lemma gives
\[(e_{m,g})_{< t}\psi_{m,t} = (e_{m,g})_{< t-1}\psi_{m,t-1}b(2\cdot 3^{t-1})w = 0,\]
where $w$ equals either $\mathbf{1}+b(3^{t-1})-b(2\cdot 3^{t-1})$ if factor $t-1$ equals $\binom{0}{0},$ or $b(2\cdot 3^{t-1})-b(3^{t-1})$ if factor $t-1$ equals $\binom{2}{1}.$
The result therefore holds in this case. 

If $x_{t-2} = 1,$ then the first statement of Lemma \ref{lem: coltom} implies that factor $t-1$ of $B(m,g)_3$ equals either $\binom{1}{0}$ or $\binom{2}{2}.$ 
By construction
\[(e_{m,g})_{< t} = (e_{m,g})_{< t-1}w,\]
where $w$ equals either $\mathbf{1}-b(2\cdot 3^{t-1})$ if factor $t-1$ equals $\binom{1}{0},$ or $b(2\cdot 3^{t-1})$ if factor $t-1$ equals $\binom{2}{2}.$ 
Then
\begin{align*}
(e_{m,g})_{< t}\psi_{m,t} &= (e_{m,g})_{< t-1}w\psi_{m,t-1}b(2\cdot 3^{t-1})\\ 
&= (e_{m,g})_{< t-1}wb(2\cdot 3^{t-1}),
\end{align*}
where the second equality holds by the inductive hypothesis of this lemma. 
If factor $t-1$ of $B(m,g)_3$ equals $\binom{1}{0},$ then the second statement of Lemma \ref{lem: coltom} applied with $m_{t-1} = 0, g_{t-1} = 0,$ and $x_{t-2} = 1$ gives $x_{t-1} = 0.$
Moreover, $w = \mathbf{1}-b(2\cdot 3^{t-1})$ in this case, and so $(e_{m,g})_{< t}\psi_{m,t} = (e_{m,g})_{< t}(\mathbf{1}-w).$
As $v := (e_{m,g})_{< t} = (e_{m,g})_{< t-1}w$ is an idempotent by the inductive hypothesis of Proposition \ref{prop: emgu}, it follows from Lemma \ref{lem: multonend} that
\[(e_{m,g})_{< t}\psi_{m,t} = v(\mathbf{1}-w) = 0.\]
If factor $t-1$ of $B(m,g)_3$ equals $\binom{2}{2},$ then the second statement of Lemma \ref{lem: coltom} now applied with $m_{t-1} = 0, g_{t-1} = 2,$ and $x_{t-2} = 1$ gives $x_{t-1} = 1.$
Moreover, $w = b(2\cdot 3^{t-1})$ in this case, and so $(e_{m,g})_{< t}\psi_{m,t} = (e_{m,g})_{< t}w.$
As $v := (e_{m,g})_{< t} = (e_{m,g})_{< t-1}w$ is an idempotent by the inductive hypothesis of Proposition \ref{prop: emgu}, it follows from Lemma \ref{lem: multonend} that
\[(e_{m,g})_{< t}\psi_{m,t} = vw = v = (e_{m,g})_{< t}.\]

{\it Case} (2). Suppose that $m_{t-1} = 1.$
Then \eqref{eq: psimtrecursion} becomes
\[\psi_{m,t} = \psi_{m,t-1}(b(3^{t-1}) + b(2\cdot 3^{t-1}))+b(2\cdot 3^{t-1}).\]

If $x_{t-2} = 0,$ then the first statement of Lemma \ref{lem: coltom} implies that factor $t-1$ of $B(m,g)_3$ equals either $\binom{1}{0}$ or $\binom{2}{2}.$ 
Again by the construction of $e_{m,g}$ 
\[(e_{m,g})_{< t} = (e_{m,g})_{< t-1}w,\]
where $w$ equals either $\mathbf{1}-b(2\cdot 3^{t-1})$ if factor $t-1$ equals $\binom{1}{0},$ or $b(2\cdot 3^{t-1})$ if factor $t-1$ equals $\binom{2}{2}.$
Moreover, the inductive hypothesis of this lemma implies that 
\[(e_{m,g})_{< t}\psi_{m,t} = (e_{m,g})_{< t-1}b(2\cdot 3^{t-1})w,\]
for both possibilities of $w.$ 
The argument is now the same as when $x_{t-2} = 1$ in Case (1). 

If $x_{t-2} = 1,$ then the first statement of Lemma \ref{lem: coltom} implies that factor $t-1$ of $B(m,g)_3$ equals either $\binom{2}{0}$ or $\binom{1}{1}.$ 
By construction
\[(e_{m,g})_{< t} = (e_{m,g})_{< t-1}w,\]
where $w$ equals either $\mathbf{1}-b(3^{t-1})+b(2\cdot 3^{t-1})$ if factor $t-1$ equals $\binom{2}{0},$ or $b(3^{t-1})-b(2\cdot 3^{t-1})$ if factor $t-1$ equals $\binom{1}{1}.$ 
Then
\begin{align*}
(e_{m,g})_{< t}\psi_{m,t} &= (e_{m,g})_{< t-1}w(\psi_{m,t-1}(b(3^{t-1}) + b(2\cdot 3^{t-1}))+b(2\cdot 3^{t-1}))\\ 
&= (e_{m,g})_{< t-1}w(b(3^{t-1}) - b(2\cdot 3^{t-1})),
\end{align*}
where the second equality holds by the inductive hypothesis of this lemma. 
If factor $t-1$ of $B(m,g)_3$ equals $\binom{2}{0},$ then the second statement of Lemma \ref{lem: coltom} applied with $m_{t-1} = 1, g_{t-1} = 0,$ and $x_{t-2} = 1$ gives $x_{t-1} = 0.$
Moreover, $w = \mathbf{1}-b(3^{t-1}) + b(2\cdot 3^{t-1})$ in this case, and so $(e_{m,g})_{<t}\psi_{m,t} = (e_{m,g})_{<t}(\mathbf{1}-w).$
As $v := (e_{m,g})_{< t} = (e_{m,g})_{< t-1}w$ is an idempotent by the inductive hypothesis of Proposition \ref{prop: emgu}, it follows from Lemma \ref{lem: multonend} that
\[(e_{m,g})_{< t}\psi_{m,t} = v(\mathbf{1}-w) = 0.\]
If factor $t-1$ of $B(m,g)_3$ equals $\binom{1}{1},$ then the second statement of Lemma \ref{lem: coltom} now applied with $m_{t-1} = 1, g_{t-1} = 1,$ and $x_{t-2} = 1$ gives $x_{t-1} = 1.$
Moreover, $w = b(3^{t-1})-b(2\cdot 3^{t-1})$ in this case.
As $v := (e_{m,g})_{< t} = (e_{m,g})_{< t-1}w$ is an idempotent by the inductive hypothesis of Proposition \ref{prop: emgu}, it follows from Lemma \ref{lem: multonend} that
\[(e_{m,g})_{< t}\psi_{m,t} = vw = v = (e_{m,g})_{< t}.\]

{\it Case} (3). Suppose that $m_{t-1} = 2.$
Then \eqref{eq: psimtrecursion} becomes
\[\psi_{m,t} = \psi_{m,t-1}(\mathbf{1}-b(3^{t-1}) + b(2\cdot 3^{t-1}))+b(3^{t-1})-b(2\cdot 3^{t-1}).\]

If $x_{t-2} = 0,$ then the first statement of Lemma \ref{lem: coltom} implies that factor $t-1$ of $B(m,g)_3$ equals either $\binom{2}{0}$ or $\binom{1}{1}.$ 
Again by the construction of $e_{m,g}$ 
\[(e_{m,g})_{< t} = (e_{m,g})_{< t-1}w,\]
where $w$ equals either $\mathbf{1}-b(3^{t-1})+b(2\cdot 3^{t-1})$ if factor $t-1$ equals $\binom{2}{0},$ or $b(3^{t-1})-b(2\cdot 3^{t-1})$ if factor $t-1$ equals $\binom{1}{1}.$
The argument is now the same as when $x_{t-2} = 1$ in Case (2). 

If $x_{t-2} = 1,$ then the first statement of Lemma \ref{lem: coltom} implies that factor $t-1$ of $B(m,g)_3$ equals either $\binom{0}{0}$ or $\binom{2}{1}.$ 
By construction
\[(e_{m,g})_{< t} = (e_{m,g})_{< t-1}w,\]
where $w$ equals either $\mathbf{1}+b(3^{t-1})-b(2\cdot 3^{t-1})$ if factor $t-1$ equals $\binom{0}{0},$ or $b(2\cdot 3^{t-1})-b(3^{t-1})$ if factor $t-1$ equals $\binom{2}{1}.$ 
Then $(e_{m,g})_{< t}\psi_{m,t}$ equals
\[(e_{m,g})_{< t-1}w(\psi_{m,t-1}(\mathbf{1}-b(3^{t-1}) + b(2\cdot 3^{t-1}))+b(3^{t-1})-b(2\cdot 3^{t-1})),\]
which by the inductive hypothesis of this lemma equals $(e_{m,g})_{< t-1}w$ for both possibilities of $w.$ 
Therefore $(e_{m,g})_{< t}\psi_{m,t} = (e_{m,g})_{< t}.$
As $m_{t-1}+x_{t-2}+g_{t-1} = 3 + g_{t-1} \ge 3,$ it follows from the second statement of Lemma \ref{lem: coltom} that $x_{t-1} = 1$ for both possible factors.
The result therefore holds in this case. 
\end{proof}
We now complete the inductive step of the proof of Proposition \ref{prop: emgu}. 
\begin{proof}[Proof of the inductive step]
Assume that $3^u \le \lambda_2 \le 2\cdot 3^u.$ 
If $\lambda_2 < 2\cdot 3^u,$ then in the following calculations we regard all terms equal to $b(2\cdot 3^u)$ as zero. 
We consider each possibility for factor $u$ of $B(m,g)_3$ in turn.  

{\it Case} (1a). 
Suppose that factor $u$ of $B(m,g)_3$ equals $\binom{2}{1}.$ 
By Lemma \ref{lem: coltom} either $m_u = 0$ and $x_{u-1} = 0,$ or $m_u = 2$ and $x_{u-1} = 1.$ 
By construction of $e_{m,g}$ and the inductive hypothesis
\begin{align*}
(e_{m,g})_{\le u}^2 &= ((e_{m,g})_{<u})^2(b(2\cdot 3^u)-b(3^u))^2\\
&= (e_{m,g})_{<u}(b(2\cdot 3^u)^2+b(2\cdot 3^u)b(3^u)+b(3^u)^2)\\
&= (e_{m,g})_{<u}b(2\cdot 3^u)\left[\binom{m_u+1}{2} + \binom{m_u+1}{1}\psi_{m,u}\right]\\
&+ (e_{m,g})_{<u}b(2\cdot 3^u)\left[2\binom{m_u}{1} - \psi_{m,u}\right]\\
&+ (e_{m,g})_{<u}\left(b(3^u)\left[\binom{m_u+2}{1} + \psi_{m,u}\right]+b(2\cdot 3^u)\right),
\end{align*}
where the final equality holds by \eqref{eq: b3sq}, \eqref{eq: b2sq} and \eqref{eq: b3b2}.
The result in this case now follows from Lemma \ref{lem: finalfactor}.

{\it Case} (1b). 
Suppose that factor $u$ of $B(m,g)_3$ equals $\binom{0}{0}.$ 
By Lemma \ref{lem: coltom} either $m_u = 0$ and $x_{u-1} = 0,$ or $m_u = 2$ and $x_{u-1} = 1.$ 
By construction of $e_{m,g}$ and the inductive hypothesis
\begin{align*}
(e_{m,g})_{\le u}^2 &= ((e_{m,g})_{<u})^2(\mathbf{1}+b(3^u)-b(2\cdot 3^u))^2\\
&= (e_{m,g})_{<u}(\mathbf{1}\!+\!b(3^u)^2\!+\!b(2\cdot 3^u)^2\! - \!b(3^u)\!+\!b(2\cdot 3^u)\!+\!b(2\cdot 3^u)b(3^u))\\
&= (e_{m,g})_{<u}(\mathbf{1}- b(3^u)+b(2\cdot 3^u))\\
&+(e_{m,g})_{<u}b(2\cdot 3^u)\left[\binom{m_u+1}{2} + \binom{m_u+1}{1}\psi_{m,u}\right]\\
&+ (e_{m,g})_{<u}b(2\cdot 3^u)\left[2\binom{m_u}{1} - \psi_{m,u}\right]\\
&+ (e_{m,g})_{<u}\left(b(3^u)\left[\binom{m_u+2}{1} + \psi_{m,u}\right]+b(2\cdot 3^u)\right),
\end{align*}
where the final equality holds by \eqref{eq: b3sq}, \eqref{eq: b2sq} and \eqref{eq: b3b2}.
Again the result in this case now follows from Lemma \ref{lem: finalfactor}. 

{\it Case} (2a). 
Suppose that factor $u$ of $B(m,g)_3$ equals $\binom{2}{2}.$ 
By Lemma \ref{lem: coltom} either $m_u = 1$ and $x_{u-1} = 0,$ or $m_u = 0$ and $x_{u-1} = 1.$ 
By construction of $e_{m,g}$ and the inductive hypothesis
\begin{align*}
(e_{m,g})_{\le u}^2 &= (e_{m,g})_{<u}^2b(2\cdot 3^u)^2\\
&= (e_{m,g})_{<u}b(2\cdot 3^u)\left[\binom{m_u+1}{2} + \binom{m_u+1}{1}\psi_{m,u}\right],
\end{align*}
where the final equality holds by \eqref{eq: b2sq}.
The result in this case now follows from Lemma \ref{lem: finalfactor}.

{\it Case} (2b). 
Suppose that factor $u$ of $B(m,g)_3$ equals $\binom{1}{0}.$ 
By Lemma \ref{lem: coltom} either $m_u = 1$ and $x_{u-1} = 0,$ or $m_u = 0$ and $x_{u-1} = 1.$ 
By construction of $e_{m,g}$ and the inductive hypothesis
\begin{align*}
(e_{m,g})_{\le u}^2 &= (e_{m,g})_{<u}^2(\mathbf{1}-b(2\cdot 3^u))^2\\
&= (e_{m,g})_{<u}^2(\mathbf{1}+ b(2\cdot 3^u) + b(2\cdot 3^u)^2)\\
&= (e_{m,g})_{<u}\left(\mathbf{1}+b(2\cdot 3^u)\left[1+\binom{m_u+1}{2} + \binom{m_u+1}{1}\psi_{m,u}\right]\right),
\end{align*}
where the final equality holds by \eqref{eq: b2sq}.
Again the result in this case now follows from Lemma \ref{lem: finalfactor}.

{\it Case} (3a). 
Suppose that factor $u$ of $B(m,g)_3$ equals $\binom{1}{1}.$ 
By Lemma \ref{lem: coltom} either $m_u = 2$ and $x_{u-1} = 0,$ or $m_u = 1$ and $x_{u-1} = 1.$ 
By construction of $e_{m,g}$ and the inductive hypothesis
\begin{align*}
(e_{m,g})_{\le u}^2 &= (e_{m,g})_{<u}^2(b(3^u)-b(2\cdot 3^u))^2\\
&= (e_{m,g})_{<u}b(2\cdot 3^u)\left[\binom{m_u+1}{2} + \binom{m_u+1}{1}\psi_{m,u}\right]\\
&+ (e_{m,g})_{<u}b(2\cdot 3^u)\left[2\binom{m_u}{1} - \psi_{m,u}\right]\\
&+ (e_{m,g})_{<u}\left(b(3^u)\left[\binom{m_u+2}{1} + \psi_{m,u}\right]+b(2\cdot 3^u)\right),
\end{align*}
where the final equality holds by \eqref{eq: b3sq}, \eqref{eq: b2sq} and \eqref{eq: b3b2}.
The result in this case now follows from Lemma \ref{lem: finalfactor}.

{\it Case} (3b). 
Suppose that factor $u$ of $B(m,g)_3$ equals $\binom{2}{0}.$ 
By Lemma \ref{lem: coltom} either $m_u = 2$ and $x_{u-1} = 0,$ or $m_u = 1$ and $x_{u-1} = 1.$ 
By construction of $e_{m,g}$ and the inductive hypothesis
\begin{align*}
(e_{m,g})_{\le u}^2 &= (e_{m,g})_{<u}^2(\mathbf{1}-b(3^u)+b(2\cdot 3^u))^2\\
&= (e_{m,g})_{<u}(\mathbf{1}\!+\!b(3^u)^2\!+\!b(2\cdot 3^u)^2\!+\!b(3^u)\!-\!b(2\cdot 3^u)\!+\!b(2\cdot 3^u)b(3^u))\\
&= (e_{m,g})_{<u}(\mathbf{1}+ b(3^u)-b(2\cdot 3^u))\\
&+(e_{m,g})_{<u}b(2\cdot 3^u)\left[\binom{m_u+1}{2} + \binom{m_u+1}{1}\psi_{m,u}\right]\\
&+ (e_{m,g})_{<u}b(2\cdot 3^u)\left[2\binom{m_u}{1} - \psi_{m,u}\right]\\
&+ (e_{m,g})_{<u}\left(b(3^u)\left[\binom{m_u+2}{1} + \psi_{m,u}\right]+b(2\cdot 3^u)\right),
\end{align*}
where the final equality holds by \eqref{eq: b3sq}, \eqref{eq: b2sq} and \eqref{eq: b3b2}.
Again the result in this case now follows from Lemma \ref{lem: finalfactor}.
\end{proof}
\begin{remark}
Given $t \in \N,$ we can generalise the definition of $\psi_{m,t}$ when $p$ is an arbitrary prime.
Furthermore, the recursive formula in \eqref{eq: psimtrecursion} generalises in an entirely similar way.
However defining $e_{m,g}$ using the $p$-adic expansion of the binomial coefficient $B(m,g)$ when $p \ge 5$ remains unknown in general.
\end{remark}
\subsection{The elements $e_{m,g}$ are orthogonal and primitive.}\label{sec: primitive} 
Let $g,d \in \N_0$ be such that both $B(m,g)$ and $B(m,d)$ are non-zero modulo 3, and suppose that $g \neq d.$ 
Write 
\begin{align*}
g &=_p [g_0,g_1,g_2,\ldots,g_t]\\ 
d &=_p [d_0,d_1,d_2,\ldots,d_t].
\end{align*}
Let $u$ be minimal such that $g_u \neq d_u,$ and so $(m+2g)_{<u} = (m+2d)_{< u}$ and $(e_{m,g})_{<u} = (e_{m,d})_{<u}.$ 
As in \S \ref{sec: carries}, let $x_{u-1}$ (resp.~$y_{u-1}$) denote the carry leaving column $u-1$ in the ternary addition of $m$ and $g$ (resp.~$d$), recalling that the columns in both $p$-ary additions are indexed starting from 0. 
It follows that $x_{u-1} = y_{u-1},$ and so $(m_u,x_{u-1}) = (m_u,y_{u-1}).$
By Lemma \ref{lem: coltom}, factor $u$ of $B(m,g)_3$ equals $\binom{a}{g_u}$ and factor $u$ of $B(m,d)_3$ equals $\binom{b}{d_u},$ where $a - 2g_u \equiv_3 b-2d_u \equiv_3 m_u + x_{u-1}.$
Moreover, these factors are unequal  since $g_u \neq d_u.$ 
As there are exactly two choices for a factor $\binom{x}{y}$ such that $0 \le y \le x < 3$ and $x-2y \equiv_3 m_u + x_{u-1},$ 
it follows from the construction of $e_{m,g}$ that
\[(e_{m,g})_{\le u} = (e_{m,g})_{< u}w \mbox{ and } (e_{m,d})_{\le u} = (e_{m,d})_{<u}(\mathbf{1}-w),\] 
where $w, \mathbf{1}-w$ are as specified in \textsection \ref{sec: notation}. 
By Proposition \ref{prop: emgu}, $(e_{m,g})_{\le u}$ and $(e_{m,d})_{\le u}$ are idempotents in $S_\F(\lambda),$ and so it follows from Lemma \ref{lem: multonend} that their product is zero.
As $S_\F(\lambda)$ is commutative, this implies $e_{m,g}e_{m,d} = 0.$ 

We now count the number of non-zero $e_{m,g}$ in $S_\F(\lambda).$
By Lemma \ref{lem: smallestinemg}, $e_{m,g}$ is non-zero in $S_\F(\lambda)$ if and only if $g \le \lambda_2.$
Therefore the number of non-zero $e_{m,g}$ in $S_\F(\lambda)$ is equal to 
\[\size{\{g : g \le \lambda_2 \mbox{ and } B(m,g) \mbox{ is non-zero modulo }3\}}.\]
By Theorem 3.3 in \cite{H} this equals the number of indecomposable summands of $M^{\lambda}.$  
It therefore follows that the set of $e_{m,g}$ such that $g \le \lambda_2$ is a complete set of primitive orthogonal idempotents for $S_\F(\lambda).$

\section{The correspondence between idempotents and Young modules}\label{sec: proof2}
Throughout this section let $\lambda = (\lambda_1,\lambda_2)$ and $\mu = (\mu_1,\mu_2)$ be partitions of $r$ satisfying the hypothesis of Theorem \ref{thm: correspondence}. 

We prove Theorem \ref{thm: correspondence} by induction on $r$ by following \cite[\S 7]{DEH}.
The base cases are $r = 0$ and $r = 1.$ 
In both cases the only possibility is $\lambda = \mu = (r,0).$
Therefore in \S \ref{sec: correspondencebase} we consider the case when $\mu = (r,0)$ and $\lambda \in \Lambda(2,r)$ is arbitrary. 
We then complete the inductive step in \S \ref{sec: correspondenceinductive}.

\subsection{The case $\mu = (r,0)$}\label{sec: correspondencebase}
We distinguish two cases determined by $\lambda.$

If $\lambda = (r,0),$ then $M^{(r,0)}$ is indecomposable and the only primitive idempotent in $S_\F((r,0))$ is $\mathbf{1}.$
In this case $B(m,g) = \binom{r}{0},$ and so 
\[B(m,g)_3 = \binom{r_0}{0}\ldots \binom{r_t}{0},\]
where $r =_3 [r_0,\ldots, r_t].$
By construction, for some $\alpha_i \in \F_3,$
\[e_{r,0} = \mathbf{1} + \sum_{i > 0} \alpha_i b(i) = \mathbf{1} \in S_\F((r,0)),\]
as required. 
Observe that this proves the base case of the induction.

Recall from \textsection \ref{sec: maintheorems} that $1_\lambda$ is an idempotent in $S_\F(2,r)$ such that $1_\lambda E^{\otimes r} = M^\lambda.$
If $\lambda = (m+g,g) \vdash r,$ then we show that there exist $u, v \in S_\F(2,r)$ such that $uv = e_{m,g}$ and $vu = 1_{(r,0)}.$
Then $e_{m,g}$ and $1_{(r,0)}$ are idempotents such that $e_{m,g} = u1_{(r,0)}v$ and $1_{(r,0)} = ve_{m,g}u.$
It follows from \cite[(1.1)]{Scott} that $e_{m,g}M^\lambda = e_{m,g}E^{\otimes r} \cong 1_{(r,0)}\nat = M^{(r,0)} = Y^{(r,0)},$ 
as required. 
Now define 
\[u = B(m,g) 1_{\lambda}f^{(g)}1_{(r,0)} \mbox{ and }v = 1_{(r,0)}e^{(g)}1_\lambda.\]
In order to calculate $uv$ and $vu$, we follow parts (b) and (c) in the proof of \cite[Proposition 7.2]{DEH}. 
Indeed define the \emph{simple root $\alpha = (1,-1)$}. 
By Theorem 2.4 in \cite{DG2} if $\nu \in \Lambda(2,r)$, then 
\begin{align*}
e\,1_\nu &= \begin{cases}
1_{\nu + \alpha} \ e & \text{ if $\nu + \alpha$ is a composition,}\\
0 & \text{ otherwise }
\end{cases}\\
f\,1_\nu &= \begin{cases}
1_{\nu - \alpha} \ f & \text{ if $\nu - \alpha$ is a composition,}\\
0 & \text{ otherwise.}
\end{cases}
\end{align*}
Moreover, Proposition 4.3 in \cite{DG2} states that $H_i 1_\lambda = \lambda_i 1_\lambda$ for $i \in \{1,2\}.$
Define $h = H_1 - H_2,$  and so $h1_\lambda = m\, 1_\lambda$. 
Since $(r,0) + (1,-1)$ is not a composition, the above relations give $e^{(a)}\, 1_{(r,0)} = 0$ for all $a \in \N.$ 
Also with $\lambda = (m+g,g),$ we have
\begin{equation}\label{eq: relations}
\begin{array}{c c c}
e^{(g)}\, 1_\lambda = 1_{(r,0)}\, e^{(g)}, & 1_{(r,0)}\, f^{(g)} = f^{(g)} \, 1_\lambda, & \binom{h}{g}\, 1_{(r,0)} = \binom{r}{g} \, 1_{(r,0)}. 
\end{array}
\end{equation}
It follows from the relations in \eqref{eq: relations} and Lemma \ref{lem: smallestinemg} that
\begin{align*}
uv &= B(m,g)\, 1_\lambda f^{(g)}1_{(r,0)}e^{(g)}1_\lambda\\
&= B(m,g)\, 1_\lambda f^{(g)}e^{(g)}1_\lambda\\ 
&= B(m,g)\, b(g) = e_{m,g} \in S_\F((m+g,g)).
\end{align*}
Also it follows from the relations in \eqref{eq: relations} and \cite[\S 26.2]{Humphreys} that 
\begin{align*}
vu &= B(m,g) 1_{(r,0)}e^{(g)}1_\lambda f^{(g)}1_{(r,0)}\\
&= B(m,g) 1_{(r,0)}e^{(g)}f^{(g)}1_{(r,0)}\\
&= 1_{(r,0)}[\sum_{j=0}^{g} f^{(g-j)}{\textstyle\binom{h-2g+2j}{j}}\displaystyle e^{(g-j)}]1_{(r,0)}\\
&= 1_{(r,0)}[f^{(0)}\binom{h}{g}e^{(0)}]1_{(r,0)}\\
&= \binom{r}{g} 1_{(r,0)} = (B(m,g))^2 1_{(r,0)} \equiv_3 1_{(r,0)},
\end{align*}
where the congruence holds as $B(m,g)$ is non-zero modulo 3.
\subsection{The inductive step}\label{sec: correspondenceinductive}
Assume throughout this section that the statement of Theorem \ref{thm: correspondence} holds inductively for all partitions in $\Lambda(2,r)$ for some $r \in \N_0.$
Let $\widetilde{\lambda}$ and $\widetilde{\mu}$ be partitions of $r+2$ with at most two parts satisfying the hypothesis of the theorem.
The argument for the case when $\widetilde{\mu}_2 = 0$ is given in \S \ref{sec: correspondencebase}, so assume that $\widetilde{\mu}_2 > 0.$
Then $\widetilde{\lambda} = \lambda + (1^2)$ and $\widetilde{\mu} = \mu+(1^2),$ where $\lambda$ and $\mu$ are the partitions of $r$ such that $m = \lambda_1 - \lambda_2$ and $g = \lambda_2-\mu_2.$ 
The inductive step is complete once we prove Proposition \ref{prop: endgame} below, which is equivalent to Theorem 7.3 in \cite{DEH}. 
To this end define the map $j : \nat \rightarrow E^{\otimes r+2}$ by 
\[x \mapsto (v_1 \otimes v_2 - v_2 \otimes v_1)\otimes x,\]
where we remind the reader that $\{v_1,v_2\}$ is a fixed basis of $E.$ 
Observe that $j$ is injective.
Also it follows from the definition of $M^\lambda$ given in \S \ref{sec: maintheorems} that $j(M^\lambda) \subset M^{\lambda + (1^2)}.$  
We then have the following lemma.
\begin{lemma}\label{lem: commutingj}
Given $x \in M^\lambda,$ we have $je_{m,g}(x) = e_{m,g}j(x).$
\end{lemma}
\begin{proof}
We prove that $jb(a)(x) = b(a)j(x)$ for all $x \in M^\lambda$ and $a \in \N_0.$  
Note that on the left hand side of this equality $b(a)$ is viewed as an element of $S_{\F}(\lambda),$ and on the right hand side it is viewed as an element of $S_{\F}(\lambda+(1^2)).$ 

The Lie algebra action of $e$ on $v_1 \otimes v_2 - v_2 \otimes v_1$ is as follows:
\begin{align*}
e (v_1 \otimes v_2 - v_2 \otimes v_1) &= (ev_1\otimes v_2 + v_1\otimes ev_2) - (ev_2 \otimes v_1 + v_2\otimes ev_1)\\ 
&= v_1\otimes v_1 - v_1\otimes v_1\\ 
&= 0.
\end{align*}
Similarly $f(v_1 \otimes v_2 - v_2 \otimes v_1) = 0,$ and so $j$ commutes with the action of $e^{(a)}$ and $f^{(a)}$ for all $a \in \N.$ 
Also considering the Lie algebra action of the product $f^{(a)}e^{(a)}$ on $M^{\lambda}$ and $M^{\lambda+{(1^2)}}$ shows that $f^{(a)}e^{(a)}$ preserves $M^{\lambda}$ and $M^{\lambda + (1^2)}$. 
As $1_\lambda$ and $1_{\lambda+(1^2)}$ are the projections onto $E^{\otimes r}$ corresponding to $M^\lambda$ and $M^{\lambda+(1^2)}$, respectively, it follows that
\begin{align*}
j(b(a)x) &= j(1_\lambda f^{(a)}e^{(a)} 1_\lambda x)\\
&= j(f^{(a)}e^{(a)} x)\\
&=  (v_1 \otimes v_2 - v_2 \otimes v_1) \otimes f^{(a)}e^{(a)} x,
\end{align*}
and 
\begin{align*}
b(a)j(x) &= (1_{\lambda+(1^2)} f^{(a)}e^{(a)} 1_{\lambda+(1^2)})((v_1 \otimes v_2 - v_2 \otimes v_1) \otimes x)\\
&= 1_{\lambda+(1^2)}(0 + (v_1 \otimes v_2 - v_2 \otimes v_1) \otimes f^{(a)}e^{(a)} x)\\
&=  (v_1 \otimes v_2 - v_2 \otimes v_1) \otimes f^{(a)}e^{(a)} x.
\end{align*}
Therefore $jb(a)x = b(a)j(x),$ as required.
\end{proof}
Before we state and prove Proposition \ref{prop: endgame}, we introduce the following notation.
Given $\boldsymbol{i} = (i_1,i_2,\ldots,i_r) \in I(2,r),$ define $v_{\boldsymbol{i}} = v_{i_1} \otimes v_{i_2} \otimes \cdots \otimes v_{i_r}.$ 
\begin{proposition}\label{prop: endgame}
Suppose that $e_{m,g}M^\lambda \cong Y^\mu.$ 
Then we have $e_{m,g}M^{\lambda + (1^2)} \cong Y^{\mu + (1^2)}$. 
\end{proposition}
\begin{proof}
Since $M^\lambda$ has a direct summand isomorphic to $Y^\mu,$ there is a submodule of $M^\lambda$ isomorphic to the Specht module $S^\mu.$
Moreover, $\F$ is a field of characteristic 3 and both $\lambda$ and $\mu$ have at most two parts, and so Theorem 13.13 in \cite{J} implies that $\Hom_{\F S_r}(S^\mu, M^{\lambda})$ is one-dimensional.
Equivalently the copy of $S^\mu$ in $M^{\lambda}$ is unique,
and the analogous statement holds for $S^{\mu + (1^2)}$ and $M^{\lambda + (1^2)}$ by the same argument. 
By the defining property of the Young module $Y^\mu,$ it sufficient to prove that if $e_{m,g}(S^\mu) \neq 0,$ then $e_{m,g}(S^{\mu+(1^2)}) \neq 0.$
We do this using polytabloids (see \cite[Chapter 4]{J} for details).

Write $u$ for $\mu_2+1.$ 
Let $t_1$ and $t_2$ respectively denote the following standard $\mu$ and $\mu+(1^2)$-tableaux:
$$t_1 = \scalebox{0.9}{\ytableausetup
{mathmode, boxsize=2em}
\begin{ytableau}
\scriptstyle 3 &\scriptstyle 5 & \none[\dots]
& \scriptstyle 2u - 1 &\scriptstyle 2u + 1 & \none[\dots]
&\scriptstyle r+2\\
\scriptstyle 4 &\scriptstyle 6& \none[\dots]
& \scriptstyle 2u
\end{ytableau}}\hspace{15pt}t_2= \scalebox{0.9}{\begin{ytableau}
\scriptstyle 1 &\scriptstyle 3 & \none[\dots]
& \scriptstyle 2u - 1 &\scriptstyle 2u + 1 & \none[\dots]
&\scriptstyle r+2\\
 \scriptstyle 2 & \scriptstyle 4& \none[\dots]
& \scriptstyle 2u
\end{ytableau}}.$$
Write $R_{t_i}$ for the row stabiliser of each $t_i.$
Also write $C_{t_i}$ for the column stabiliser of each $t_i,$ and define $\{C_{t_i}\}^- = \sum_{\pi \in C_{t_i}} \sgn(\pi)\pi.$ 
It is easy to see that $\{C_{t_2}\}^- = (1 - (1\ 2))\{C_{t_1}\}^-.$  

Observe that the column stabiliser of $t_1$ is a subgroup of the symmetric group on $\{3,4,\ldots,r+2\}.$ 
Thus given $\sigma \in \Sym(\{3,4,\ldots,r+2\}),$ we define $\sigma^\star \in \Sym(\{1,2,\ldots,r\})$ to be the permutation such that $\sigma^\star(\ell) = \sigma(\ell+2)-2$ for $1 \le \ell \le r.$
Then there is a natural action of $\sigma \in \Sym(\{3,4,\ldots,r+2\})$ on $x \in M^\lambda$ given by $x\sigma = x\sigma^\star.$ 

Let $\omega_1 = \sum v_{\boldsymbol{i}},$ where the sum runs over all $\boldsymbol{i} \in I(2,r)$ such that $\boldsymbol{i}$ has weight $\lambda$ and $i_{\rho} = 2$ whenever $\rho+2$ is in the second row of $t_1.$
Observe that $\omega_1$ is fixed by $R_{t_1},$
and so define the \emph{polytabloid} $\varepsilon_{t_1} = \omega_1 \{C_{t_1}\}^-.$
Note that the actions of $R_{t_1}$ and $\{C_{t_1}\}^-$ on $\omega_1$ are as defined in the previous paragraph.
Then $\varepsilon_{t_1}$ generates the unique copy of $S^\mu$ in $M^\lambda.$

Similarly let $\omega_2 =  \sum v_{\boldsymbol{i}},$ where the sum runs over all $\boldsymbol{i} \in I(2,r+2)$ such that $\boldsymbol{i}$ has weight $\lambda$ and $i_\rho = 2$ whenever $\rho$ is in the second row of $t_2.$ 
Again $\omega_2$ is fixed by $R_{t_2},$ and so 
define the polytabloid $\varepsilon_{t_2} = \omega_2 \{C_{t_2}\}^-.$
Note that the actions of $R_{t_2}$ and $\{C_{t_2}\}^-$ on $\omega_2$ are given by the usual place permutation defined in \S \ref{sec: maintheorems}. 
Then $\varepsilon_{t_2}$ generates the unique copy of $S^{\mu + (1^2)}$ in $M^{\lambda+(1^2)}.$ 
By definition of $j$  
\begin{align*}
j(\varepsilon_{t_1}) &=  (v_1\otimes v_2 - v_2 \otimes v_1)\otimes\varepsilon_{t_1} \\
&= ((v_1\otimes v_2 - v_2 \otimes v_1)\otimes \omega_1) \{C_{t_1}\}^-,
\end{align*}
where in the final line the action of $\{C_{t_1}\}^-$ is again by the usual place permutation defined in \S \ref{sec: maintheorems}.

Observe that 
\[\omega_2 = v_1 \otimes v_2 \otimes \omega_1 + v_2\otimes v_2 \otimes \omega,\]
where $\omega$ is the sum of the $v_{\boldsymbol{i}}$ such that $\boldsymbol{i} \in I(2,r)$ has weight $\lambda + (1,-1),$ and $i_\rho = 2$ whenever $\rho+2$ is in the second row of $t_1.$
(In the case that $\lambda_2 = 0,$ we have $\omega = 0.$)
Then since $v_1\otimes v_2 - v_2 \otimes v_1 = (v_1 \otimes v_2)(1 - (1\ 2)),$ we have
\begin{align*}
j(\varepsilon_{t_1}) &= ((v_1\otimes v_2 - v_2 \otimes v_1)\otimes \omega_1) \{C_{t_1}\}^-\\
&=  (v_1\otimes v_2\otimes \omega_1)(1 - (1\ 2))\{C_{t_1}\}^-\\
&= (v_1 \otimes v_2 \otimes \omega_1 + v_2 \otimes v_2 \otimes \omega)(1 - (1\ 2))\{C_{t_1}\}^-\\
&= {\omega}_2\{C_{t_2}\}^- = \varepsilon_{t_2},
\end{align*}
where the third equality holds since $(v_2 \otimes v_2 \otimes \omega)(1 - (1\ 2)) = 0.$ 

If $e_{m,g}(S^\mu)\neq 0$, then $e_{m,g}(\varepsilon_{t_1})\neq 0$. 
The map $j$ is injective and so $j(e_{m,g}(\varepsilon_{t_1}))\neq 0.$ 
It now follows from Lemma \ref{lem: commutingj} that 
\[e_{m,g}(\varepsilon_{t_2}) = e_{m,g}(j(\varepsilon_{t_1})) = j(e_{m,g}(\varepsilon_{t_1})) \neq 0,\] 
and so $e_{m,g}(S^{\mu+(1^2)}) \neq 0.$
Therefore $e_{m,g}(Y^{\mu+(1^2)}) \neq 0,$ which completes the proof.
\end{proof}
\section*{Acknowledgements} 
This paper was completed by the author under the supervision of Professor Mark Wildon.
The author gratefully acknowledges his support. 
The author also thanks an anonymous referee for their careful reading of an earlier version of this paper. 
The author was funded by EPSRC grant EP/M507945/1. 
\bibliographystyle{amsplain}
\bibliography{../thesis/thesis}
\end{document}